\documentclass[12pt,a4paper,psamsfonts]{amsart}
\usepackage[english]{babel}
\usepackage[T1]{fontenc}
\usepackage{fancyhdr}
\usepackage{appendix}
\usepackage{amssymb,amscd,amsxtra,calc}
\usepackage{mathrsfs}
\usepackage{cmmib57}
\usepackage{multirow}
\usepackage[all]{xy}
\usepackage{longtable}
\usepackage[colorlinks=true,linkcolor=blue,anchorcolor=blue,citecolor=blue,pagebackref,linktocpage]{hyperref}
\usepackage{cleveref}
\usepackage{tikz}
\usepackage{cite}
\theoremstyle{plain}
    \newtheorem{thm}{Theorem}[section]

     \newtheorem{conjecture}[thm]{Conjecture}

    \newtheorem{lemma}[thm]{Lemma}
    \newtheorem{proposition}[thm]{Proposition}
    
     \newtheorem{problem}[thm]{Problem}
    \newtheorem{theorem}[thm]{Theorem}

\theoremstyle{definition}
    \newtheorem{definition}[thm]{Definition}

    \newtheorem*{notation*}{Notation and Terminology}
      
    \newtheorem{remark}[thm]{Remark}
    
\theoremstyle{remark}

\newcommand{\arxiv}[1]{\href{https://arxiv.org/abs/#1}{{\tt arXiv:#1}}}

\newcommand{\bP}{\mathbb{P}}

\newcommand{\bR}{\mathbb{R}}
\newcommand{\bC}{\mathbb{C}}

\newcommand{\bZ}{\mathbb{Z}}

\newcommand{\NE}{\overline{\operatorname{NE}}}

\newcommand{\Pic}{\operatorname{Pic}}

\newcommand{\mstriangle}[1]{
\begin{tikzpicture}[x=0.3cm,y=0.3cm]
\draw (-0.4,-0.433) -- (1.4,-0.433);
\draw (-0.2,-0.7794) -- (0.7,0.7794);
\draw (1.2,-0.7794) -- (0.3,0.7794);
\end{tikzpicture}
}
\newcommand{\mssharp}[1]{
\begin{tikzpicture}[x=0.3cm,y=0.3cm]
\draw (-0.8,-0.5) -- (0.8,-0.5);
\draw (-0.8,0.5) -- (0.8,0.5);
\draw (-0.5,-0.8) -- (-0.5,0.8);
\draw (0.5,-0.8) -- (0.5,0.8);
\end{tikzpicture}
}

\makeatletter

\newcommand{\Rmnum}[1]{\expandafter\@slowromancap\romannumeral #1@}
\makeatother

\begin{document}
\title[Bounded cohomology property]
{Smooth projective surfaces with bounded cohomology property II}
\author{Sichen Li}
\address{
School of Mathematics, East China University of Science and Technology, Shanghai 200237, P. R. China}
\email{\href{mailto:sichenli@ecust.edu.cn}{sichenli@ecust.edu.cn}}

\begin{abstract}
Let $X$ be a smooth projective surface with Picard number two, where either $X$ is a geometrically ruled surface or the closed Mori cone is rational polyhedral.
In this paper, we characterize $X$ with the bounded cohomology property, i.e., there exists a constant $c_X>0$ such that $h^1(\mathcal O_X(C))\le c_Xh^0(\mathcal O_X(C))$ for every curve $C$ on $X$.
\end{abstract}
\maketitle
\section{Introduction}
A very open problem in the theory of projective surfaces is the following Bounded Negativity Conjecture (BNC for short).
\begin{conjecture}
\cite[Conjecture 1.1]{Bauer et al. 2013}
For a smooth projective surface $X$ over $\bC$, there exists an integer $b(X)\ge0$ such that $C^2\ge-b(X)$ for every curve $C$ on $X$.
\end{conjecture}
It is well-known that the following SHGH Conjecture implies Nagata's Conjecture \cite{Nagata59}, which is motivated by Hilbert's 14th problem (cf. \cite[Lemma 2.4]{CHMR13}).
\begin{conjecture}
\cite[Conjecture 2.5.1]{Bauer et al. 2012} 
Let $C\subseteq X$ be a curve where $X\to\bP^2$ is the blow-up of general points $p_1,\cdots, p_n$ with $n\ge10$.
Then $h^1(\mathcal O_X(C))=0$.
\end{conjecture}
Motivated by the BNC and the SHGH Conjecture, a fundamental problem is to classify smooth projective surfaces $X$ with the bounded cohomology property, which is due to Bauer et al. (cf. \cite[Conjecture 2.5.3]{Bauer et al. 2012}), as follows.
\begin{problem}
(cf. \cite[Question 6]{Ciliberto et al. 2017})
\label{Pro-BCP}
Classify all smooth projective surfaces $X$ with the Bounded cohomology property (BCP for short), i.e., there exists a constant $c_X>0$ such that $h^1(\mathcal O_X(C))\le c_Xh^0 (\mathcal O_X(C))$ for every curve $C$ on $X$.
\end{problem}
\begin{remark}
Ciliberto et al. \cite{Ciliberto et al. 2017} showed that the BCP implies the BNC.
For recent progress of the BCP, we refer to \cite{Bauer et al. 2012, Ciliberto et al. 2017, HL26, Li21, Li23}.
For some examples of surfaces with the BCP, we refer to \cite[Section 5]{HL26}.
\end{remark}
Ciliberto et al. showed in  \cite[Proposition 15]{Ciliberto et al. 2017} that if there exists a constant \(c_X>0\) satisfying \(h^1(\mathcal O_X(C))\le c_X h^0(\mathcal O_X(C))\) for every curve $C$ with \(C^2\ge0\), then there exists another constant \(m(X)>0\) such that \(l_C\le m(X)\) for all such curves.
As proved in the main results of \cite{HL26, Li21, Li23}, a primary strategy to establish the BCP is to show the uniform boundedness of $l_C$ (see  Definition \ref{Defn-l_C}).
Let $X$ be a geometrically ruled surface with invariant $e$ over a smooth curve $C$ of genus $g$, determined by a normalized locally free sheaf $\mathcal E$.
Note that the closed Mori cone $\NE(X)$ may not be polyhedral provided that $e<0$ (cf. \cite[Example 1.5.1]{Lazarsfeld04}).
Recently, Ziyu Hua and the author established in \cite[Theorem 1.10]{HL26} that  $X$ admits the uniform boundedness of $l_C$.
In this paper,  we first characterize geometrically ruled surfaces with the BCP as follows.
\begin{theorem}
\label{Thm-Ruled}
Let $X$ be a geometrically ruled surface with invariant $e$ over a smooth curve $C$ of genus $g$, determined by a normalized locally free sheaf $\mathcal E$.
Let $C_0\subseteq X$ be a section such that $\mathcal O_X(C_0)\cong O_X(1)$ and $C_0^2=-e$, and let $f$ be a fibre.
Then the following statements hold.
\begin{enumerate}
	\item If $e\ne0$, then $X$ satisfies the BCP.
	\item  If $e=0$ and $\mathcal E$ is decomposable, then $X$ satisfies the BCP.
	\item If $e=0$ and $\mathcal E$ is indecomposable, then $X$ satisfies the BCP if and only if $$\liminf_{a\to+\infty}\frac{h^0(\mathcal O_X(aC_0+bf))}{a}>0$$ for each $b\in (0,g-1)\cap\bZ$.
	\end{enumerate}
\end{theorem}
Let $X$ be a smooth projective surface with Picard number $\rho(X)=2$ such that $\NE(X)$ is rational polyhedral.
If the Kodaira dimension $\kappa(X)=-\infty$, by the Enriques-Kodaira classification of relatively minimal surfaces (cf. \cite{BHPV04, Beauville96}),  $X$ is either a ruled surface or a point blow-up of $\bP^2$.
Note that $X$ satisfies the BCP if it is a point blow-up of $\bP^2$ (cf. \cite[Lemma 3.2]{Li23}).
To characterize $X$ with the BCP, by \cite[Propositions 2.6 and 2.7]{HL26} and Theorem \ref{Thm-Ruled},  it suffices to consider the essential case of $\kappa(X)=2$ as follows.
\begin{theorem}
\label{kappa2-thm}
Let $X$ be a smooth projective surface of general type.
Suppose $\NE(X)=\bR_{\ge0}[C_1]+\bR_{\ge0}[C_2]$, and $C_1$ and $C_2$ are curves on $X$.
Then $X$ satisfies the BCP if and only if the following statements hold.
\begin{enumerate}
	\item we may assume that  $K_X=aC_1+a'C_2$, $D=a_1C_1+a_2C_2$, with $a,a',a_1,a_2\in\bR_{>0}$, and either (i) $C_1^2=C_2^2=0$, or (ii) $C_1^2=0, C_2^2<0$.
	\item If $C_1^2=0$ and $C_2^2=0$, then either $$\liminf_{a_1\to+\infty}\frac{h^0(\mathcal O_X(a_1C_1+a_2C_2))}{a_1}>0, \text{ with } a_1>a, \text{ and } 0<2a_2<a'$$
	or $$\liminf_{a_2\to+\infty}\frac{h^0(\mathcal O_X(a_1C_1+a_2C_2))}{a_2}>0, \text{ with } 0<2a_1<a \text{ and } a_2>a'.$$
	\item If $C_1^2=0$ and $C_2^2<0$, then 
$$\liminf_{a_1\to+\infty}\frac{h^0(\mathcal O_X(a_1C_1+a_2C_2))}{a_1}>0, \text{ with } a_1>a \text{ and } 0<2a_2<a'.$$ 
\end{enumerate}
\end{theorem}
\begin{remark}
In this paper, we study the BCP for smooth projective surfaces $X$ with $\rho(X)=2$, where either $X$ is a geometrically ruled surface or $\NE(X)$ is rational polyhedral.
The classification of the BCP for surfaces with higher Picard numbers remains a widely open problem.
In this setting, one is generally forced to assume that $\NE(X)$ is rational polyhedral.
For instance, the authors of \cite{HL26} established the BCP for Mori dream surfaces, for which $q(X)=0$, $\overline{\operatorname{NE}}(X)$ is rational polyhedral and every nef divisor is semiample.
Their proof hinges on the condition that 
\begin{equation*}
\label{eq:key-condition}
\operatorname{Nef}(X)=\sum \mathbb R_{\ge 0}[D_i] \text{ and every } D_i \text{ satisfies } \kappa(X,D_i)\ge 1
\end{equation*}
which holds automatically for Mori dream surfaces.
\end{remark}

 {\bf Acknowledgment.}
The author would like to thank the referees for their valuable comments and suggestions.
The research is supported by the Shanghai Sailing Program (No. 23YF1409300).
\section{Proof of Theorems \ref{Thm-Ruled} and \ref{kappa2-thm}}
\label{Pre}
{\bf Notation and Terminology.}
In this paper, $X$ is a smooth projective surface over $\bC$.
\begin{itemize}
\item By a curve on $X$, we mean a reduced and irreducible curve.
\item A negative curve on $X$ is a curve with negative self-intersection.
 \item A prime divisor $C$ on $X$ is either a nef curve or a negative curve (in the latter case, $h^0(\mathcal O_X(C))=1$).
 \item For every $\bR$-divisor $C$ with $C^2\ne0$ on $X$, we define a value $l_C$ associated to $C$ as follows:
\begin{equation*}
                                                       l_C:=\frac{(K_X\cdot C)}{\max\bigg\{ 1, C^2\bigg\}}.
\end{equation*}
\item For every $\bR$-divisor $C$ with $C^2=0$ on $X$, we define a value $l_C$ associated to $C$ as follows:
\begin{equation*}
                                   l_C:=\frac{(K_X\cdot C)}{\max\bigg\{1,h^0(\mathcal O_X(C))\bigg\}}.
\end{equation*}
\end{itemize}
\begin{definition}
\cite[Definition 2.1]{HL26}
 \label{Defn-l_C}
We say a smooth projective surface $X$ admits uniform boundedness of $l_C$ if there exists a positive constant $m(X)$ such that $l_C\le m(X)$ for every curve $C$ on $X$.
\end{definition}
\begin{proposition}
\label{Prop-Numerical}
(cf. \cite[Proposition 2.3]{Li21})
Let $X$ be a smooth projective surface.
If $X$ admits the uniform boundedness of $l_C$, and there exists a positive constant $m(X)$ such that either $|C^2|\le m(X)h^0(\mathcal O_X(C))$ or $h^1(\mathcal O_X(C))\le m(X)h^0(\mathcal O_X(C))$ for every curve $C$ on $X$, then $X$ satisfies the BCP.
\end{proposition}
The following result is due to Serre duality.
\begin{proposition}\label{Serre}
Let $C$ be a curve on a smooth projective surface $X$.
Then
\begin{equation*} 
 h^2(\mathcal O_X(C))-\chi(\mathcal O_X)\le q(X)-1.	
\end{equation*}
\end{proposition}
From now on, we will consistently use the definitions and notation of geometrically ruled surfaces as presented in \cite[Chapter V.2]{Hartshorne77}.
\begin{proposition}
\label{Prop-ruled}
\cite[Propositions V.2.3 and V.2.9]{Hartshorne77}
Let $\pi: X\to C$ be a geometrically ruled surface with invariant $e$ over a smooth curve $C$ of genus $g$.
Let $C_0\subseteq X$ be a section such that $\mathcal O_X(C_0)\cong O_X(1)$ and $C_0^2=-e$, and let $f$ be a fibre.
Then we have the following result:
$$
	\Pic 	X\cong \bZ C_0\oplus\pi^*\Pic C, C_0\cdot f=1, f^2=0, \text{~and ~}~K_X\equiv -2C_0+(2g-2-e)f.
$$
\end{proposition}
The following result is due to Biancofiore and Livorni \cite[p.175, (0.8)]{BL87}.
\begin{proposition}
\label{BL-Prop}
Let $X$ be a geometrically ruled surface with invariant $e$ over a smooth curve $C$ of genus $g$.
Then for every curve $D=aC_0+bf$ as in Proposition \ref{Prop-ruled}, we have
$$
 h^1(\mathcal O_X(D))=h^0(\mathcal O_X(D))+\bigg(\frac{ae}{2}+g-1-b\bigg)(a+1).	
$$
\end{proposition}
Now we first simplify the proof of \cite[Lemma 3.1]{Li23} as follows.
\begin{proposition}
\label{e>0-prop}
Let $X$ be a geometrically ruled surface with invariant $e>0$ over a smooth curve $C$ of genus $g$.
Then $X$ satisfies the BCP.	
\end{proposition}
\begin{proof}
To show the BCP for $X$, we may take a curve $D=aC_0+bf$ (in which $D\neq C_0,f$) as in Proposition \ref{Prop-ruled}.
Then $a>0$ and $b\ge ae$ by \cite[Proposition V.2.20]{Li23}.
By Proposition \ref{BL-Prop}, we have
 \begin{equation}
 \label{BL-eq}
 h^1(\mathcal O_X(D))=h^0(\mathcal O_X(D))+\bigg(\frac{ae}{2}+g-1-b\bigg)(a+1).	
 \end{equation}
if $b\ge \frac{ae}{2}+g-1$, then $h^1(\mathcal O_X(D))\le h^0(\mathcal O_X(D))$ by (\ref{BL-eq}).
If $b<\frac{ae}{2}+g-1$, then $a<\frac{2g-2}{e}$ since $ae\le b$.
As a result, by (\ref{BL-eq}) we also have
\begin{equation*}
\begin{split}
h^1(\mathcal O_X(D))&=h^0(\mathcal O_X(D))+\bigg(\frac{ae}{2}+g-1-b\bigg)(a+1)
\\&\le h^0(\mathcal O_X(D))+2(g-1)(2g-2+e)e^{-1}
\\&\le\max\{1,1+2(g-1)(2g-2+e)e^{-1}\} h^0(\mathcal O_X(D)).
\end{split}
\end{equation*}
Thus, $X$ satisfies the BCP.
\end{proof}
\begin{lemma}
\label{lem-e<0}
Let $X$ be a geometrically ruled surface with invariant $e<0$ over a smooth curve $C$ of genus $g$.
Then $X$ satisfies the BCP.
\end{lemma}
\begin{proof}
By \cite[Theorem 1.10]{HL26}, $X$ admits the uniform boundedness of $l_C$, and we may assume that $g\ge2$.
Note that $X$ satisfies the BNC since $\rho(X)=2$.
Then by Proposition \ref{Prop-Numerical}, it suffices to show that there is a constant $c_X>0$ such that $h^1(\mathcal O_X(D))\le c_Xh^0(\mathcal O_X(D))$ for every curve $D$ with $D^2>0$ on $X$.
Now we may take a curve $D=aC_0+bf$ (in which $D\neq C_0,f$) as in Proposition \ref{Prop-ruled}.
Then either $a=1, b>0$ or $a\ge2, b\ge\frac{ae}{2}$ by \cite[Proposition V.2.21]{Hartshorne77}.

If $b\ge\frac{ae}{2}+g-1$, then $h^1(\mathcal O_X(D))\le h^0(\mathcal O_X(D))$ by Proposition \ref{BL-Prop}.
Now we may assume that $b<\frac{ae}{2}+g-1$.
Note that $K_X\equiv -2C_0+(2g-2-e)f$ by Proposition \ref{Prop-ruled}.

(Case I). We assume that $a=1$ and $b>0$.
Then $0<b<g+\frac{e}{2}-1$.
From this together with $e<0$, it follows that $g>1+\frac{e}{2}$.
Note that $$0<-e<C_0\cdot D=b-e<g-\frac{e}{2}-1, \quad f\cdot D=1.$$
As a result, we have 
\begin{equation}
\begin{split}
\label{b>0-e<0-eq}
	(K_X-D)\cdot D&=-3(C_0\cdot D)+(2g-2-e-b)(f\cdot D)
	\\&\le 2g-2-e.
\end{split}
\end{equation}
Therefore, by the Riemann-Roch theorem, Proposition \ref{Serre} and (\ref{b>0-e<0-eq}), we have
\begin{equation*}
\begin{split}
h^1(\mathcal O_X(D))&=h^0(\mathcal O_X(D))+\frac{(K_X-D)\cdot D}{2}+h^2(\mathcal O_X(D))-\chi(\mathcal O_X)
\\&\le (1+3(g-1)-e)h^0(\mathcal O_X(D)),
\end{split}	
\end{equation*}
where we use $q(X)=g$ by \cite[Proposition V.2.5]{Hartshorne77}.

(Case II). We assume that $a\ge2$ and $b\ge\frac{ae}{2}$.
Note that $D^2=0$ if and only if  $b=\frac{ae}{2}$.
So we have 
\begin{equation}
\label{b-value-eq}
\bigg(\frac{ae}{2}+g-1\bigg)-b\in (0,g-1)\cap\bZ,
\end{equation}
 and $D$ is ample by \cite[Proposition V.2.21]{Hartshorne77}.
As a result, $\kappa(X,D)=2$ and $D^2\ge1$.
Fix $a_0$ and $b_0$ such that $$D_0=a_0C_0+b_0f, \quad  a_0>0,  \quad b_0\in\bigg(\frac{a_0e}{2}, \frac{a_0e}{2}+g-1\bigg)\cap\bZ, \quad D=mD_0. $$
By \cite[Corollary 1.4.41]{Lazarsfeld04} and $\kappa(X,D)=2$, there is a positive integer $N\gg0$, for each $m\ge N$, we have
\begin{equation}
\label{kappa2-e<0-eq}
	h^0(\mathcal O_X(mD_0))=\frac{m^2D_0^2}{2}+O(m).
\end{equation}
Then by Proposition \ref{BL-Prop}, (\ref{b-value-eq}) and (\ref{kappa2-e<0-eq}), we have
\begin{equation*}
\begin{split}
	h^1(\mathcal O_X(D))&=h^1(\mathcal O_X(mD_0))
	\\&=h^0(\mathcal O_X(mD_0))+\bigg(\bigg(\frac{ae}{2}+g-1\bigg)-b\bigg)(ma_0+1)
	\\&\le h^0(\mathcal O_X(mD_0)+(g-1)(ma_0+1)
	\\&\le 2h^0(\mathcal O_X(mD_0))
	\\&=2h^0(\mathcal O_X(D)),
\end{split}
\end{equation*}
where we use  $$\frac{m^2D_0^2}{2}+O(m)>(g-1)(ma_0+1)$$ since $m\ge N$ and $N\gg0$.

If $m\le N$, then $a\le a_0N$.
By Proposition \ref{BL-Prop}, we have
\begin{equation*}
\begin{split}
h^1(\mathcal O_X(D))&\le h^0(\mathcal O_X(D))+(g-1)(a+1)
\\&\le (1+(g-1)(a_0N+1))h^0(\mathcal O_X(D)).
\end{split}
\end{equation*}
In conclusion, $X$ satisfies the BCP.
\end{proof}
\begin{proposition}
\label{reduced-e=0-prop}
Let $X$ be a geometrcially ruled surface with invariant $e=0$ over a smooth curve $C$ of genus $g$.
To show the BCP for $X$,it suffices to show that there is a constant $c_X>0$ such that $h^1(\mathcal O_X(D))\le c_Xh^0(\mathcal O_X(D))$ for every ample curve $D=aC_0+bf$ with $a>0$ and $0<b<g-1$.
\end{proposition}
\begin{proof}
By \cite[Theorem 1.10]{HL26}, $X$ admits the uniform boundedness of $l_C$, and we may assume that $g\ge2$.
Note that $X$ satisfies the BNC since $\rho(X)=2$.
By Proposition \ref{Prop-Numerical}, it suffices to show that there is a constant $c_X>0$ such that $h^1(\mathcal O_X(D))\le c_Xh^0(\mathcal O_X(D))$ for every curve $D=aC_0+bf$ with $D^2>0$ on $X$.
Note that $\NE(X)=\bR_{\ge0}[C_0]+\bR_{\ge0}[f]$.
Then $a>0$ and $b>0$.
As a result, $D$ is ample by \cite[Proposition V.2.20]{Hartshorne77}.
If $b\ge g-1$, then $h^1(\mathcal O_X(D))\le h^0(\mathcal O_X(D))$ by  Proposition \ref{BL-Prop}.
Now we may assume that $0<b< g-1$ and $g\ge2$.
This ends the proof of Proposition \ref{reduced-e=0-prop}.
\end{proof}
\begin{lemma}
\label{e=0-deco-lem}
Let $X$ be a geometrically ruled surface with invariant $e=0$ over a smooth curve $C$ of genus $g$, determined by a normalized locally free sheaf $\mathcal E$.
If $\mathcal E$ is decomposable, then $X$ satisfies the BCP.
\end{lemma}
\begin{proof}
To show the BCP for $X$, we may take a curve $D=aC+bf$ with $a>0$ and $0<b< g-1$ by Proposition \ref{reduced-e=0-prop}.
As noted in \cite[Notation 2.8.1]{Hartshorne77}, let $\mathfrak e$ be the divisor on $C$ corresponding to the invertible sheaf $\wedge^2\mathcal E$, so that $\deg\mathfrak e=-e=0$.
Now let $D=aC_0+\mathfrak bf$ for some divisor $\mathfrak b$ on $C$ with $\deg\mathfrak b=C_0\cdot D=b$.
By \cite[Lemma 4]{FP05}, we have the following result:
\begin{equation}
\label{h1h0-eq1}
	h^i(\mathcal O_X(D))=(a+1)h^i(\mathcal O_C(\mathfrak b)) \text{ for } i\in\{0,1\}.
\end{equation}
Since $D$ is a curve, $h^0(\mathcal O_X(D))>0$.
By  Proposition \ref{BL-Prop}, we have
\begin{equation}
\label{h1h0-eq2}
h^1(\mathcal O_X(D))\ge h^0(\mathcal O_X(D))>0, \quad h^1(\mathcal O_C(\mathfrak b))\ge h^0(\mathcal O_X(\mathfrak b))>0.
\end{equation}
So by (\ref{h1h0-eq2}), we have
\begin{equation}
\label{c_X-eq}
c_X:=\max_{0<\deg\mathfrak b<g-1}\frac{h^1(\mathcal O_C(\mathfrak b))}{h^0(\mathcal O_C(\mathfrak b))}>0,
\end{equation}
which is independent of the choice of $D$.
As a result, by (\ref{h1h0-eq1})  and (\ref{c_X-eq}), we have
$$
h^1(\mathcal O_X(D))\le c_Xh^0(\mathcal O_X(D)).
$$
Therefore, $X$ satisfies the BCP.
\end{proof}
\begin{lemma}
\label{liminf-lem}
Let $X$ be a geometrically ruled surface with invariant $e=0$ over a smooth curve $C$ of genus $g$, determined by a normalized locally free sheaf $\mathcal E$.
If $e=0$ and $\mathcal E$ is indecomposable, then $X$ satisfies the BCP if and only if  $$\liminf_{a_1\to+\infty}\frac{h^0(\mathcal O_X(aC_0+bf))}{a}>0$$ for  each $b\in (0,g-1)\cap\bZ$.
\end{lemma}
\begin{proof}
To show the BCP for $X$, we may take a curve $D=aC_0+bf$ with $a>0$ and $0<b< g-1$ by Proposition \ref{reduced-e=0-prop}.
By Proposition \ref{BL-Prop}, we have
\begin{equation}
\label{indecom-e=0-eq1}
 h^1(\mathcal O_X(D))=h^0(\mathcal O_X(D))+(g-1-b)(a+1).
\end{equation}
($\Rightarrow$) Fix $b\in(0,g-1)$.
Since $X$ satisfies the BCP, there is a positive constant $c_X>1$ such that 
\begin{equation}
\label{c_X}
h^1(\mathcal O_X(D))\le c_Xh^0(\mathcal O_X(D)).	
\end{equation} 
By (\ref{indecom-e=0-eq1}) and (\ref{c_X}), 
we have
$$
(c_X-1)h^0(\mathcal O_X(D))\ge (g-1-b)(a+1).
$$
Therefore, we have
\begin{equation}
\label{inf-lim>0}
\liminf_{a\to\infty}\frac{h^0(\mathcal O_X(aC_0+bf))}{a}>0.
\end{equation}

($\Leftarrow$) For each $b\in (0,g-1)$, (\ref{inf-lim>0}) implies that
 there exist a positive integer $N_b$ and a constant $k_b>0$ such that for $a>N_b$, we have
\begin{equation}
\label{indecom-e=0-eq3}
k_bh^0(\mathcal O_X(aC_0+bf))\ge (g-1-b)(a+1).
\end{equation}
As a result, if $a>N_b$, then by (\ref{indecom-e=0-eq1}) and (\ref{indecom-e=0-eq3}), we have
$$
h^1(\mathcal O_X(D))\le (k_b+1)h^0(\mathcal O_X(D)).
$$
If $a\le N_b$, then by (\ref{indecom-e=0-eq1}), we have
$$
h^1(\mathcal O_X(D))\le ((g-1)(N_b+1)+1)h^0(\mathcal O_X(D)).
$$
Now we take $c_X:=\max_{0<b<g-1}\{k_b+1, (g-1)(N_b+1)+1\}$.
Then for every curve $D=aC_0+bf$ with $a>0, 0<b<g-1$, we have
$$
h^1(\mathcal O_X(D))\le c_Xh^0(\mathcal O_X(D)).
$$
Thus, $X$ satisfies the BCP by Proposition  \ref{reduced-e=0-prop}.
\end{proof}
\begin{proof}[Proof of Theorem \ref{Thm-Ruled}]
If $e\ne0 $, then $X$ satisfies the BCP by Proposition \ref{e>0-prop} and Lemma \ref{lem-e<0}.
If $e=0$ and $\mathcal E$ is decomposable, then $X$ satisfies the BCP by Lemma \ref{e=0-deco-lem}.
If $e=0$ and $\mathcal E$ is indecomposable, then the proof follows from Lemma \ref{liminf-lem}.
\end{proof}
\begin{proof}[Proof of Theorem \ref{kappa2-thm}]
Since $\kappa(X)=2$, we may assume that $K_X=aC_1+a'C_2$ with $a,a'>0$.
If $C_1^2<0$ and $C_2^2<0$, then $X$ satisfies the BCP by \cite[Lemma 3.4]{Li21}.
So we may assume that either (i) $C_1^2=C_2^2=0$ or (ii) $C_1^2=0, C_2^2<0$.
Note that $X$ satisfies the BCP since $\rho(X)=2$.
By \cite[Proposition 3.4]{Li23} and Proposition \ref{Prop-Numerical}, it suffices to show that there exists a positive constant $m(X)>1$ such that either $h^1(\mathcal O_X(D))\le m(X)h^0(\mathcal O_X(D))$ or $D^2\le m(X)h^0(\mathcal O_X(D))$ for every curve $D$ with $D^2>0$ on $X$.
We may take $D=a_1C_1+a_2C_2$ with $a_1,a_2>0$ by \cite[Proposition 3.1]{Li21}.

{\bf Case I: $C_1^2=0$ and $C_2^2=0$.}
If $a_1>a$ and $a_2>a'$, then $D-K_X=(a_1-a)C_1+(a_2-a')C_2$ is big.
Note that $C_1\cdot C_2>0$.
Then $$(D-K_X)\cdot C_1=(a_2-a')(C_1\cdot C_2)>0, \quad (D-K_X)\cdot C_2=(a_1-a)(C_1\cdot C_2)>0.$$
Therefore, $D-K_X$ is nef and big since $\NE(X)=\bR_{\ge0}[C_1]+\bR_{\ge0}[C_2]$.
By Kawamata-Viehweg vanishing theorem, we have 
$$h^1(\mathcal O_X(D))=h^1(\mathcal O_X(K_X+(D-K_X)))=0.$$

If $a_1\le a, a_2\le a'$, then we have 
\begin{equation*}
\begin{split}
		D^2 &=2a_1a_2(C_1\cdot C_2) 
		\\ & \le 2aa'(C_1\cdot C_2)
		\\ &\le  (2aa'(C_1\cdot C_2))h^0(O_X(D)).
\end{split}
\end{equation*}

Now we may assume that either (i) $a_1>a$ and $a_2\le a'$, or (ii) $a_2\le a$ and $a_1>a'$.
Without loss of generality, we may assume that $a_1>a$ and $a_2\le a'$.
Note that 
\begin{equation}
\begin{split}
\label{K_XD-eqI}
(K_X-D)\cdot D=(a_1(a'-2a_2)+aa_2)(C_1\cdot C_2).
\end{split}
\end{equation}
By the Riemann-Roch theorem, we have
\begin{equation}
\begin{split}
\label{RR-eqI}
h^1(\mathcal O_X(D))&=h^0(\mathcal O_X(D))+\frac{(K_X-D)\cdot D}{2}+\chi(\mathcal O_X)-h^2(\mathcal O_X(D))
\end{split}
\end{equation}
If $a'-2a_2\le0$, then by (\ref{K_XD-eqI}), (\ref{RR-eqI}) and Proposition \ref{Serre}, we have
\begin{equation}
\begin{split}
\label{RR-eqII}
h^1(\mathcal O_X(D))&=h^0(\mathcal O_X(D))+\frac{(K_X-D)\cdot D}{2}+\chi(\mathcal O_X)-h^2(\mathcal O_X(D))
\\&\le h^0(\mathcal O_X(D))+\frac{aa'(C_1\cdot C_2)}{2}+q(X)-1.
\\&\le \bigg(1+q(X)+\frac{aa'(C_1\cdot C_2)}{2}\bigg)h^0(\mathcal O_X(D)).
\end{split}
\end{equation}
So we may assume that $a'-2a_2>0$.
Then by (\ref{K_XD-eqI}), we have
\begin{equation}
\label{K_X-lim-eq}
\liminf_{a_1\to +\infty}\frac{(K_X-D)\cdot D}{a_1}>0.
\end{equation}

Now we claim that for  every curve $D=a_1C_1+a_2C_2$ with $a_1>a$ and $a'-2a_2>0$, 
\begin{equation}
\label{BCP-eqI}
	h^1(\mathcal O_X(D))\le c_Xh^0(\mathcal O_X(D)), \quad c_X>1.
\end{equation}
 if and only if 
 \begin{equation}
 \label{I-Lim-eqI}
 	\liminf_{a_1\to+\infty}\frac{h^0(\mathcal O_X(a_1C_1+a_2C_2))}{a_1}>0.
 \end{equation}
 
 ($\Rightarrow$)  (\ref{RR-eqI}), (\ref{K_X-lim-eq}) and (\ref{BCP-eqI}) imply that
 $$
 	\liminf_{a_1\to+\infty}\frac{h^0(\mathcal O_X(a_1C_1+a_2C_2))}{a_1}\ge \frac{1}{c_X-1}\liminf_{a_1\to +\infty}\frac{(K_X-D)\cdot D}{a_1}>0.
 $$
 
($\Leftarrow$)   (\ref{I-Lim-eqI}) implies that there is a positive  constant $N$ and a constant $k>0$ such that for $a_1>N$, we have 
\begin{equation}
\label{BCP-RR-II}
kh^0(\mathcal O_X(a_1C_1+a_2C_2))\ge \frac{a'(C_1\cdot C_2)}{2}a_1+\frac{aa_2(C_1\cdot C_2)}{2}+\chi(\mathcal O_X)-h^2(\mathcal O_X(D)).
\end{equation}
So by (\ref{K_XD-eqI}), (\ref{RR-eqI}), and (\ref{BCP-RR-II}), we have
$$
h^1(\mathcal O_X(D))\le (k+1)h^0(\mathcal O_X(D)).
$$
If $a_1\le N$, then by  (\ref{K_XD-eqI}), (\ref{RR-eqI}) and Proposition \ref{Serre}, we have
$$
h^1(\mathcal O_X(D))\le \bigg(1+q(X)+\frac{a'(N+a)(C_1\cdot C_2)}{2}\bigg)h^0(\mathcal O_X(D)).
$$
Then for every curve  $D=a_1C_1+a_2C_2$ with $a_1>a$ and $a'-2a_2>0$,  we have
$$
h^1(\mathcal O_X(D))\le c_Xh^0(\mathcal O_X(D)),
$$
where $$c_X:=\max\bigg\{k+1, 1+q(X)+\frac{a'(N+a)(C_1\cdot C_2)}{2}\bigg\}.$$
Thus, $X$ satisfies the BCP.

{\bf Case II. $C_1^2=0$ and $C_2^2<0$}.
By \cite[Proposition 3.5]{Li23}, we may assume that $a_1>a$ and $0<2a_2<a'$.
Note that $D=a_1C_1+a_2C_2$ and 
\begin{equation}
\begin{split}
\label{K_XD-eqII}
(K_X-D)\cdot D&=(a_1(a'-2a_2)+aa_2)(C_1\cdot C_2)+a_2(a_2-a')(-C_2^2).
\end{split}
\end{equation}
Observe that the negative contribution in \eqref{K_XD-eqII} satisfies
\begin{equation*}
\begin{split}
|a_2(a_2-a')(-C_2^2)|&=a_2(a'-a_2)(-C_2^2)
\\&\le \frac{a'^2}{2}(-C_2^2).
\end{split}
\end{equation*}
The right-hand side is a constant.
Therefore, we have
\begin{equation*}
\liminf_{a_1\to+\infty}\frac{(K_X-(a_1C_1+a_2C_2))\cdot(a_1C_1+a_2C_2)}	{a_1}>0.
\end{equation*}
Therefore, using the same argument as in Case I,   $X$ satisfies the BCP if and only if 
\begin{equation*}
\liminf_{a_1\to+\infty}\frac{h^0(\mathcal O_X(a_1C_1+a_2C_2))}{a_1}>0,
\end{equation*}
with $0<2a_2<a'$.
\end{proof}

\end{document}